\documentclass[12pt]{article}
\usepackage[all]{xy}
\usepackage{amsthm, amsmath, amssymb, enumerate}
\usepackage{fullpage}
\usepackage[colorlinks=true]{hyperref}

\begin{document}
\title{A note related to Severi's finiteness theorem}
\author{Guoquan Gao}
\maketitle

\theoremstyle{plain}
\newtheorem{thm}{Theorem}[section]
\newtheorem{theorem}[thm]{Theorem}
\newtheorem{cor}[thm]{Corollary}
\newtheorem{corollary}[thm]{Corollary}
\newtheorem{lem}[thm]{Lemma}
\newtheorem{lemma}[thm]{Lemma}
\newtheorem{pro}[thm]{Proposition}
\newtheorem{proposition}[thm]{Proposition}
\newtheorem{prop}[thm]{Proposition}
\newtheorem{definition}[thm]{Definition}
\newtheorem{assumption}[thm]{Assumption}
\newtheorem{conjecture}[thm]{conjecture}

\theoremstyle{remark}
\newtheorem{remark}[thm]{Remark}
\newtheorem{example}[thm]{Example}
\newtheorem{remarks}[thm]{Remarks}
\newtheorem{problem}[thm]{Problem}
\newtheorem{exercise}[thm]{Exercise}
\newtheorem{situation}[thm]{Situation}
\newtheorem{Question}[thm]{Question}

\numberwithin{equation}{subsection}

\newcommand{\ZZ}{\mathbb{Z}}
\newcommand{\CC}{\mathbb{C}}
\newcommand{\QQ}{\mathbb{Q}}
\newcommand{\RR}{\mathbb{R}}
\newcommand{\HH}{\mathcal{H}}     % upper half plane

\newcommand{\ad}{\mathrm{ad}}            % _ad admissible class
\newcommand{\NT}{\mathrm{NT}}
\newcommand{\nonsplit}{\mathrm{nonsplit}}
\newcommand{\Pet}{\mathrm{Pet}}
\newcommand{\Fal}{\mathrm{Fal}}

\newcommand{\cs}{{\mathrm{cs}}}

\newcommand{\ZZn}{\mathbb{Z}[1/n]}
\newcommand{\ZZN}{\mathbb{Z}[1/N]}

\newcommand{\pair}[1]{\langle {#1} \rangle}
\newcommand{\wpair}[1]{\left\{{#1}\right\}}
\newcommand{\wh}{\widehat}
\newcommand{\wt}{\widetilde}

\newcommand\Spf{\mathrm{Spf}}

\newcommand{\lra}{{\longrightarrow}}

\newcommand{\matrixx}[4]
{\left( \begin{array}{cc}
  #1 &  #2  \\
  #3 &  #4  \\
 \end{array}\right)}        % 2*2 matrix

%%%%%%%%%%%%%%%%%%%%%%%%%%

\newcommand{\BA}{{\mathbb {A}}}
\newcommand{\BB}{{\mathbb {B}}}
\newcommand{\BC}{{\mathbb {C}}}
\newcommand{\BD}{{\mathbb {D}}}
\newcommand{\BE}{{\mathbb {E}}}
\newcommand{\BF}{{\mathbb {F}}}
\newcommand{\BG}{{\mathbb {G}}}
\newcommand{\BH}{{\mathbb {H}}}
\newcommand{\BI}{{\mathbb {I}}}
\newcommand{\BJ}{{\mathbb {J}}}
\newcommand{\BK}{{\mathbb {K}}}
\newcommand{\BL}{{\mathbb {L}}}
\newcommand{\BM}{{\mathbb {M}}}
\newcommand{\BN}{{\mathbb {N}}}
\newcommand{\BO}{{\mathbb {O}}}
\newcommand{\BP}{{\mathbb {P}}}
\newcommand{\BQ}{{\mathbb {Q}}}
\newcommand{\BR}{{\mathbb {R}}}
\newcommand{\BS}{{\mathbb {S}}}
\newcommand{\BT}{{\mathbb {T}}}
\newcommand{\BU}{{\mathbb {U}}}
\newcommand{\BV}{{\mathbb {V}}}
\newcommand{\BW}{{\mathbb {W}}}
\newcommand{\BX}{{\mathbb {X}}}
\newcommand{\BY}{{\mathbb {Y}}}
\newcommand{\BZ}{{\mathbb {Z}}}

\newcommand{\CA}{{\mathcal {A}}}
\newcommand{\CB}{{\mathcal {B}}}
\newcommand{\CD}{{\mathcal{D}}}
\newcommand{\CE}{{\mathcal {E}}}
\newcommand{\CF}{{\mathcal {F}}}
\newcommand{\CG}{{\mathcal {G}}}
\newcommand{\CH}{{\mathcal {H}}}
\newcommand{\CI}{{\mathcal {I}}}
\newcommand{\CJ}{{\mathcal {J}}}
\newcommand{\CK}{{\mathcal {K}}}
\newcommand{\CL}{{\mathcal {L}}}
\newcommand{\CM}{{\mathcal {M}}}
\newcommand{\CN}{{\mathcal {N}}}
\newcommand{\CO}{{\mathcal {O}}}
\newcommand{\CP}{{\mathcal {P}}}
\newcommand{\CQ}{{\mathcal {Q}}}
\newcommand{\CR }{{\mathcal {R}}}
\newcommand{\CS}{{\mathcal {S}}}
\newcommand{\CT}{{\mathcal {T}}}
\newcommand{\CU}{{\mathcal {U}}}
\newcommand{\CV}{{\mathcal {V}}}
\newcommand{\CW}{{\mathcal {W}}}
\newcommand{\CX}{{\mathcal {X}}}
\newcommand{\CY}{{\mathcal {Y}}}
\newcommand{\CZ}{{\mathcal {Z}}}

\newcommand{\ab}{{\mathrm{ab}}}
\newcommand{\Ad}{{\mathrm{Ad}}}
\newcommand{\an}{{\mathrm{an}}}
\newcommand{\Aut}{{\mathrm{Aut}}}

\newcommand{\Br}{{\mathrm{Br}}}
\newcommand{\bs}{\backslash}
\newcommand{\bbs}{\|\cdot\|}

\newcommand{\Ch}{{\mathrm{Ch}}}
\newcommand{\cod}{{\mathrm{cod}}}
\newcommand{\cont}{{\mathrm{cont}}}
\newcommand{\cl}{{\mathrm{cl}}}
\newcommand{\criso}{{\mathrm{criso}}}

\newcommand{\dR}{{\mathrm{dR}}}
\newcommand{\disc}{{\mathrm{disc}}}
\newcommand{\Div}{{\mathrm{Div}}}
\renewcommand{\div}{{\mathrm{div}}}

\newcommand{\Eis}{{\mathrm{Eis}}}
\newcommand{\End}{{\mathrm{End}}}

\newcommand{\Frob}{{\mathrm{Frob}}}

\newcommand{\Gal}{{\mathrm{Gal}}}
\newcommand{\GL}{{\mathrm{GL}}}
\newcommand{\GO}{{\mathrm{GO}}}
\newcommand{\GSO}{{\mathrm{GSO}}}
\newcommand{\GSp}{{\mathrm{GSp}}}
\newcommand{\GSpin}{{\mathrm{GSpin}}}
\newcommand{\GU}{{\mathrm{GU}}}
\newcommand{\BGU}{{\mathbb{GU}}}

\newcommand{\Hom}{{\mathrm{Hom}}}
\newcommand{\Hol}{{\mathrm{Hol}}}
\newcommand{\HC}{{\mathrm{HC}}}

\renewcommand{\Im}{{\mathrm{Im}}}
\newcommand{\Ind}{{\mathrm{Ind}}}
\newcommand{\inv}{{\mathrm{inv}}}
\newcommand{\Isom}{{\mathrm{Isom}}}

\newcommand{\Jac}{{\mathrm{Jac}}}
\newcommand{\JL}{{\mathrm{JL}}}

\newcommand{\Ker}{{\mathrm{Ker}}}
\newcommand{\KS}{{\mathrm{KS}}}

\newcommand{\Lie}{{\mathrm{Lie}}}

\newcommand{\new}{{\mathrm{new}}}
\newcommand{\NS}{{\mathrm{NS}}}

\newcommand{\ord}{{\mathrm{ord}}}
\newcommand{\ol}{\overline}

\newcommand{\rank}{{\mathrm{rank}}}

\newcommand{\PGL}{{\mathrm{PGL}}}
\newcommand{\PSL}{{\mathrm{PSL}}}
\newcommand{\Pic}{\mathrm{Pic}}
\newcommand{\Prep}{\mathrm{Prep}}
\newcommand{\Proj}{\mathrm{Proj}}

\newcommand{\Picc}{\mathcal{P}ic}

\renewcommand{\Re}{{\mathrm{Re}}}
\newcommand{\red}{{\mathrm{red}}}
\newcommand{\sm}{{\mathrm{sm}}}
\newcommand{\sing}{{\mathrm{sing}}}
\newcommand{\reg}{{\mathrm{reg}}}

\newcommand{\tor}{{\mathrm{tor}}}
\newcommand{\tr}{{\mathrm{tr}}}

\newcommand{\ur}{{\mathrm{ur}}}

\newcommand{\vol}{{\mathrm{vol}}}

\newcommand{\ds}{\displaystyle}

\tableofcontents

\section{Introduction}
In this note, we consider a question (Question \ref{nor}) related to the high-dimensional generalization of the classical Severi's finiteness theorem for curves. We will introduce some background and then state the main result (Theorem \ref{mai}). For simplicity, all the varieties are defined over $\CC$ throughout this note. Moreover, all the morphisms between complex varieties are $\CC$-morphisms. \medskip

 We first recall three classical theorems of curves. A well-known theorem of Hurwitz considering the automorphisms of a curve is the following.
\begin{theorem}{\rm(Hurwitz)}\label{g>1}
If $X$ is a smooth projective curve over  $\CC $ of genus $g\geq2$, then the automorphism group $Aut(X)$ of $X$ is finite. Moreover, the order of $Aut(X)$ is at most $84(g-1)$.
\end{theorem}
A relative version of the Theorem \ref{g>1} is the following de Franchis' theorem, which is also classical.
\begin{theorem}{\rm(de Franchis)}\label{de}
Let $X,Y$ be two smooth projective curves over $\CC$, then if the genus of $Y$ is greater than 1, the set of finite morphisms from $X$ to $Y$ is finite.
\end{theorem}
We note that, once there is a finite morphism $f: X\to Y$, by the Hurwitz's formula, the genus of $X$ is greater than or equal to the genus of $Y$. Moreover, if $X,Y$ have the same genus (greater than 1), then $\deg f=1$ and $f$ must be an isomorphism. Hence we can recover Theorem \ref{g>1} if we take $X=Y$ in Theorem \ref{de}.

More generally, if we fix an $X$ and vary $Y$, we have the following finiteness theorem due to Severi.
\begin{theorem}{\rm(Severi)}\label{Sev}
Let $X$ be a smooth projective curve over $\CC$. Then up to equivalence, there are only finitely many pairs $(Y,f)$, where $Y$ is a smooth projective curve over $\CC$ with genus greater than 1, and $f:X\to Y$ is a finite morphism. Here, we say two pairs $(Y,f)$ and $(Y',f')$ are equivalent if there is an isomorphism $h: Y\to Y'$ such that $f'=h\circ f$.
\end{theorem}
Theorem \ref{Sev} implies that, up to isomorphism, there are only finitely many smooth projective curve $Y$ such that the genus $g(Y)>1$ and $Y$ is the image of some finite morphism $f:X\to Y$. Conversely, if we assume Theorem \ref{de}, this property also implies Theorem \ref{Sev}.\medskip

 Now we consider the generalization of de Franchis' Theorem \ref{de} and Severi's Theorem \ref{Sev} to the higher dimension.

 There are two high-dimensional generalizations of the curves of genus greater than 1, namely, the notion of \emph{hyperbolicity} and the notion of \emph{general type}.

There are several notions about the hyperbolicity. In this note we only use the notion \emph{Brody hyperbolicity}. A complex space $X$ is called \emph{Brody hyperbolic} if there is no nonconstant holomorphic map $h:\CC\to X$ (this kind of map is also called an \emph{entire curve} on $X$). Another very useful definition of hyperbolicity using an intrinsic metric on $X$ is due to Kobayashi \cite{Kob}, who is regarded as the founder of this field. When $X$ is projective (which always holds in this note), it is \emph{Brody hyperbolic} if and only if it is \emph{Kobayashi hyperbolic}. In the following, we just use \emph{hyperbolic} instead of \emph{Brody hyperbolic}.

The notion \emph{general type} is more standard. We say a smooth projective complex variety $X$ is of \emph{general type} if the canonical sheaf of $X$ is big. For a general projective complex variety $X$, we say $X$ is of \emph{general type} if $X$ is birational to a smooth projective variety of general type. Note that every complex variety has a desingularization due to the Hironaka's resolution of singularity. Moreover, if two smooth projective varieties $X$ and $Y$ are birational, then $X$ is of \emph{general type} if and only if $Y$ is. So the notion of \emph{general type} is a birational invariant.

Here is a high-dimensional generalization of Theorem \ref{de}. The hyperbolic case is due to Noguchi (cf. \cite[Thm. A]{Nog1}) and the general type case is due to Kobayashi and Ochiai (cf. \cite[Thm. 7.6.1]{Kob}).
\begin{theorem}{\rm (Noguchi, Kobayashi-Ochiai)}\label{Nog}
Let $X$ and $Y$ be projective varieties over $\CC$. If $Y$ is either hyperbolic or of general type, then there are only finitely many dominant rational maps from $X$ to $Y$.
\end{theorem}
Next we consider the generalization of Theorem \ref{Sev}. For this, there is a resolved Iitaka-Severi conjecture (cf. \cite[Thm. 4.1]{GP}):
\begin{theorem}\label{IS}
Let $X$ be a projective variety over $\CC$. Then up to equivalence, there are only finitely many pairs $(Y,f)$, where $Y$ is a projective variety of general type over $\CC$, and $f:X\dashrightarrow Y$ is a dominant rational map. Here, we say two pairs $(Y,f)$ and $(Y',f')$ are equivalent if there is a birational map $h: Y\dashrightarrow Y'$ such that $f'=h\circ f$.
\end{theorem}
Note that we do not need $X, Y$ to be smooth, since the equivalence in the theorem is the birational equivalence. On the other hand, if $Y, Y'$ are smooth, both hyperbolic and of general type, then $Y$ and $Y'$ are birational if and only if they are isomorphic. This follows from \cite[Cor. 6.3.10]{Kob} or the algebraic version \cite[Lem. 3.2, Lem. 3.5]{JK}. From this we can deduce the following corollary considering the above pairs up to isomorphism (instead of birational equivalence):
\begin{corollary}\label{iso}
Let $X$ be a projective variety over $\CC$, then up to equivalence, there are only finitely many pairs $(Y,f)$, where $Y$ is a \textbf{smooth} projective variety over $\CC$ that is both hyperbolic and of general type, and $f:X\to Y$ is a surjective morphism. Here, we say two pairs $(Y,f)$ and $(Y',f')$ are equivalent if there is an isomorphism $h: Y\to Y'$ such that $f'=h\circ f$.
\end{corollary}
This recovers Severi's Theorem \ref{Sev}.

In Corollary \ref{iso}, since we only modulo the isomorphisms, the smoothness condition for $Y$ is really important. For example, any smooth curve can have infinitely many contractions, all of them are birational equivalent but not isomorphic.

 The main purpose of this note is to consider the situation (in Corollary \ref{iso}) when $Y$ is normal, which is a weaker constraint than smoothness. More precisely, we ask the following question:
\begin{Question}\label{nor}
Let $X$ be a projective variety over $\CC$, then up to equivalence, are there only finitely many pairs $(Y,f)$, where $Y$ is a \textbf{normal} projective variety over $\CC$ that is both hyperbolic and of general type, and $f:X\to Y$ is a surjective morphism? Here, we say two pairs $(Y,f)$ and $(Y',f')$ are equivalent if there is an isomorphism $h: Y\to Y'$ such that $f'=h\circ f$.
\end{Question}
The goal of this note is to give a negative answer to the Question \ref{nor}. In other words, we will construct a projective surface $X$ and infinitely many pairs $(Y,f)$ satisfying the condition in Question \ref{nor}. The main result is as follows:
\begin{theorem}{\rm (main result)}\label{mai}
There is a smooth projective surface $X$ over $\CC$, with infinitely many pairs $(Y,f)$ where $Y$ is a \textbf{normal} projective surface that is both hyperbolic and of general type and $f:X\to Y$ is a surjective birational morphism. Moreover, any two pairs of them are non-equivalent in the sense of Question \ref{nor}.
\end{theorem}
The proof of Theorem \ref{mai} will take up the next section. The proof is of course constructive. This result means that it's hard to make the smoothness condition for $Y$ weaker in Corollary \ref{iso}.\medskip

 On the other hand, the following Theorem \ref{XY} implies that (in some sense) those infinitely many pairs in Theorem \ref{mai} can not form a family! The theorem was essentially proved by \cite{Nog2} and was rewritten in \cite[Theorem 2.8]{XY}.
\begin{theorem}\label{XY}
Let $K=\CC (B)$ be the function field of a smooth projective curve $B$ over $\CC$. Let $Y$ be a normal projective variety over $K$. Let $\mathcal{Y}\to B$ be an integral model of $Y$ over $B$. Assume that one of the following conditions holds:
\begin{enumerate}[(1)]
\item for some closed point $b\in B$, the fibre $\mathcal{Y}_{b}$ is hyperbolic;
\item $Y$ is a variety of general type over $K$.
\end{enumerate}
Assume that there is a smooth projective variety $X$ over $\CC$ together with a surjective $K$-morphism $\varphi:X\times_{\CC}K\to Y$. Then $Y$ is isomorphic to the base change $T\times_{\CC}K$ for a projective variety $T$ over $\CC$.
\end{theorem}
In fact, if we take $X$ to be the surface in Theorem \ref{mai}, $U$ is an open subset of $B$ and $\CY\to U$ ia a family of normal hyperbolic projective complex surfaces, and take $\Phi:X\times_{K}U\to\CY$ to be a surjective $U$- morphism which can be viewed as a family of $\CC$- morphisms from $X$ to the fibres of $\CY\to U$ satisfies the condition in Theorem \ref{mai}. Then the above Theorem \ref{XY} implies that $\CY$ is actually a (locally) constant family. This means that the infinitely many morphisms $f:X\to Y$ in Theorem \ref{mai} cannot form a family so they are a bit subtle in this sense.\medskip

\noindent\textbf{Acknowledgment.} I would like to thank my advisor Xinyi Yuan for telling me this topic about hyperbolic complex spaces. The counterexample to be constructed in this note mainly came from an idea of Chenyang Xu, who was discussing with Xinyi Yuan for the counterexample of another problem. So I also want to thank Chenyang Xu for his beautiful idea. Besides, I also want to thank Ruoyi Guo, Zeyu Wang, Junyi Xie and Qizheng Yin for the discussion of some technical issues.
\section{Construction of the counterexample} \label{sec con}
In this section, we prove Theorem \ref{mai}. In other words, we'll construct a counterexample of Question \ref{nor}, which is the normal version of Corollary \ref{iso}. We mainly use algebraic geometry of surfaces, which can be found in \cite{Har}.

Let $X_{0}$ be a smooth projective surface over $\CC$ with infinitely many $(-1)$-curves. Here, a $(-1)$-curve of $X_{0}$ is an irreducible nonsingular curve $E$ on $X_{0}$ such that $E$ is isomorphic to $\BP^{1}$ and the self-intersection number $E^{2}=-1$. Such an $X_{0}$ exists, for example, by taking $X_{0}$ as $\BP^{2}$ blowing up 9 points which are in general position. See \cite[Theorem 4a]{Nag} for more details.

 Let $E_{i} (i=1,2,\cdots)$ be infinitely many different $(-1)$-curves on $X_{0}$ as above. By Castelnuovo's contraction theorem (see \cite[V, Theorem 5.7]{Har}), for each $i\in\BN^{+}$, there exists a morphism $\tau_{i}: X_{0}\rightarrow Z_{i}$ to a smooth projective surface $Z_{i}$, and a point $P_{i}\in Z_{i}$, such that $X_{0}$ is isomorphic via $\tau_{i}$ to the monoidal transformation of $Z_{i}$ with centre $P_{i}$, and $E_{i}$ is the exceptional curve.

 We need the following key lemma:
 \begin{lemma}\label{fin}
 There exists a smooth projective hyperbolic surface $X$ over $\CC$ and a finite morphism $\pi: X\to X_{0}$ from $X$ to $X_{0}$.
 \end{lemma}
 \begin{proof}
  Let $X_{1}$ be any smooth projective complex surface which is both hyperbolic and of general type. We claim that there is a smooth projective surface $X$ with two finite morphisms to $X_{1}$ and $X_{0}$.

In fact, this is a direct consequence of a theorem of Hironaka (cf. \cite[Theorem 2.6]{Hir}) combined with Bertini's theorem. The Hironaka's theorem is as follows: If $f:X\to Y$ is a morphism between two projective complex varieties $X$ and $Y$, $L$ is an ample line bundle of $X$. Then there exists an integer $m_{0}\geq0$ such that for any $m\geq m_{0}$, the generic section of the sheaf $L^{\otimes m}$ does not vanish identically on any positive-dimensional irreducible component of any fibre of $f$.

To derive the claim at the beginning of the proof, we apply the above Hironaka's theorem to the two projection maps $p_{i}:X_{0}\times X_{1}\to X_{i} (i=0, 1)$. Let $L$ be an ample line bundle on $X_{0}\times X_{1}$. By the theorem above, for $m$ sufficient large, a generic section $s$ of $L^{\otimes m}$ does not vanish on any fibre of the two projections $p_{i} (i=0, 1)$. This means that ${\rm div}(s)$ is a divisor of $X_{0}\times X_{1}$ which intersects each fibre of $p_{i} (i=0, 1)$ properly. Combined with the Bertini's theorem, we can also guarantee that ${\rm div}(s)$ is smooth. Then repeat the above process to the two projections $p_{i}\big|_{{\rm div}(s)} (i=0, 1)$, we can find a generic section $s'$ of line bundle $(L\big|_{{\rm div}(s)})^{\otimes m'}$ (for $m'$ sufficiently large) such that ${\rm div}(s')$ is a smooth  divisor of ${\rm div}(s)$ which intersect each positive-dimensional irreducible component of each fibre of $p_{i}\big|_{{\rm div}(s)} (i=0, 1)$ properly. Then since $X_{i} (i=0, 1)$ are surfaces, the fibres of two projections $p_{i}\big|_{{\rm div}(s')} (i=0, 1)$ are all finite. This means that $p_{i}\big|_{{\rm div}(s')} (i=0, 1)$ are finite as they are projective. Thus we can take $X={\rm div}(s')$.

To prove the lemma, we just need to notice that $X$ is already hyperbolic and of general type. This is because $X$ has a finite morphism to $X_{1}$, and $X_{1}$ is hyperbolic and of general type. More precisely, the hyperbolicity of $X$ is trivial. For the property of being general type, we can use the fact that for any $n\geq 1$, the natural map $H^{0}(X_{1}, \omega_{X_{1}}^{\otimes n})\to H^{0}(X, \omega_{X}^{\otimes n})$ is injective, where $\omega_{X}, \omega_{X_{1}}$ denotes the canonical sheaf of $X, X_{1}$.

   Thus, this $X$ is what we want.
 \end{proof}
 We choose such $X$ and $\pi: X\to X_{0}$ as in the Lemma \ref{fin}. For each $i\in\BN^{+}$, consider the the composition morphism $\tau_{i}\circ\pi: X\to Z_{i}$. By Stein factorization theorem (cf., \cite[Corollary 11.5]{Har}), it can be factored into $\pi_{i}\circ f_{i}$, where $\pi_{i}:Y_{i}\to Z_{i}$ is a finite morphism and $f_{i}$ is a projective morphism with connected fibres. Moreover, $f_{i*}\mathcal{O}_{X}=\mathcal{O}_{Y_{i}}$. In other words, we have the following commutative diagram:
 \begin{displaymath}
 \xymatrix{X\ar[r]^{\pi}\ar[d]_{f_{i}}& X_{0}\ar[d]^{\tau_{i}}\\
   Y_{i}\ar[r]_{\pi_{i}}& Z_{i}}
 \end{displaymath}
 In this case, since $\tau_{i}\circ\pi: X\to Z_{i}$ is generically finite, we see that $f_{i}$ must be a birational morphism. More explicitly, let $S_{i}=\pi_{i}^{-1}(P_{i})$ be the fibre of $P_{i}$ under $\pi_{i}$, $E_{i}'=\pi^{-1}(E_{i})$, then $f_{i}$ induces an isomorphism $X\setminus E_{i}'\to Y_{i}\setminus S_{i}$. This means that the singular locus of the surface $Y_{i}$ is contained in the finite set $S_{i}$ as $X$ is smooth. Note that these $Y_{i}$ can be singular (actually, they must be singular, for otherwise, $Y_{i}$ is smooth for some $i$, then $f_{i}$ can be factored into a finite composition of monoidal transformations by \cite[V, Corollary 5.4]{Har}, so $X$ contains at least one $(-1)$-curve, but this contradicts to the fact that $X$ is hyperbolic).

 Now we can prove the following Proposition, which implies that the $X$ and $f_{i}:X\to Y_{i}$ constructed above actually gives the example in Theorem \ref{mai}.
 \begin{prop}\label{alm}
 All $Y_{i}$\rm($i\in\BN^{+}$) are normal and hyperbolic.
 \end{prop}
 \begin{proof}
 Since $f_{i*}\mathcal{O}_{X}=\mathcal{O}_{Y_{i}}$, and $X$ is smooth (hence normal), the normality of $Y_{i}$ follows from \cite[Example 2.1.15]{Laz}. Now let's prove the hyperbolicity of $Y_{i}$.

 We prove by contradiction. Assume that for some $i\in\BN^{+}$, $Y_{i}$ is not hyperbolic. Then by definition, there exists a nonconstant holomorphic map $h:\CC\to Y_{i}$. We will show that $h$ can be lifted to a nonconstant holomorphic map $\tilde{h}:\CC\to X$, which gives a contradiction since $X$ is hyperbolic.

 Since $h$ is nonconstant and $S_{i}$ is a finite set, $h(\CC)\neq S_{i}$, $h^{-1}(S_{i})$ is a (possibly empty) discrete subset of $\CC$. By the above description, $h$ induces a holomorphic map $\CC\setminus h^{-1}(S_{i})\to Y_{i}\setminus S_{i}\to X\setminus E_{i}'\to X$. Where the first map is the restriction of $h$, the second map is the inverse of $f_{i}\big|_{(X\setminus E_{i}')}$, and the third map is the inclusion. We denote this composition map $h_{0}:\CC\setminus h^{-1}(S_{i})\to X$. We are going to extend $h_{0}$ to a holomorphic map $\tilde{h}:\CC\to X$ in the following.

 For any point $z\in h^{-1}(S_{i})$ (if exists), since $X$ is compact under the complex topology, there are two possibilities: either $z$ is a removable singularity of $h_{0}$, in this case $h_{0}$ can be extended directly to $z$; or $z$ is an essential singularity of $h_{0}$, but in this case, since $h$ is nonconstant and holomorphic at $z$, for a sufficiently small neighbourhood $U$ of $z$ (under the complex topology), $h(U)$ will be contained in a sufficiently small neighbourhood of $h(z)$, so $h_{0}(U\setminus\{z\})$ will be contained in a sufficiently small open subset of $X$ (under the complex topology) as $f_{i}:X\to Y_{i}$ is finite. But on the other hand, by the big Picard's theorem, for any neighbourhood $U$ of $z$, $h_{0}(U\setminus\{z\})$ cannot be contained in a very small poen subset of $X$ (under the complex topology). This yields a contradiction.

 By the above argument, we know that all the points $z\in h^{-1}(S_{i})$ are removable singularities of $h_{0}$, thus $h_{0}$ can be extended to a holomorphic map $\tilde{h}:\CC\to X$. This ends the proof.
 \end{proof}
 \begin{remark}
 The key of the above proof is to extend the holomorphic map $h_{0}$ to an entire curve on $X$. In this case we just need to use the classical big Picard's theorem with some easy argument. But in fact, there is a much stronger result about the extension property of hyperbolic complex spaces; see \cite[Theorem (6.3.10)]{Kob} for more details.
 \end{remark}
 From Proposition \ref{alm} we can easily complete the proof of the main Theorem \ref{mai}:
 \begin{proposition}
 For $i\in \BN^{+}$, $X$ and the pairs $(Y_{i},f_{i})$ give the desired example in Theorem \ref{mai}.
 \end{proposition}
 \begin{proof}
 Since $f_{i}:X\to
 Y_{i}$ contracts the set $E_{i}'$ to a finite set $S_{i}$ of $Y_{i}$ and maps $X\setminus E_{i}'$ isomorphicly to $Y_{i}\setminus S_{i}$, any two of these pairs $(Y_{i},f_{i})$ are not equivalent in the sense of Theorem \ref{mai} as the sets $E_{i}'$ are different. Combine this fact and Proposition \ref{alm}.
 \end{proof}

\end{document}